\documentclass[10pt]{amsart}

\usepackage{amsmath, amssymb, amsbsy}
\usepackage{thmtools}
\usepackage{array}
\usepackage{graphpap, color, paralist,xcolor}
\usepackage[mathscr]{eucal}
\usepackage[pdftex]{graphicx}
\usepackage[pdftex,colorlinks,backref=page,citecolor=blue]{hyperref}

\usepackage{mathtools}
\usepackage{cleveref}
\usepackage{float}
\usepackage{crossreftools}

\newcounter{bullet}

\usepackage{setspace}
\setdisplayskipstretch{1}

\newenvironment{subproof}[1][\proofname]{
  
  \begin{proof}[#1]
}{
  \end{proof}
}

\usepackage{geometry}
\usepackage{comment}
\newtheorem{thm}{Theorem}[section]
\newtheorem{prop}[thm]{Proposition}
\newtheorem{cor}[thm]{Corollary}
\newtheorem{lem}[thm]{Lemma}

\newtheorem{conj}[thm]{Conjecture}

\theoremstyle{definition}
\newtheorem{mydef}[thm]{Definition}

\newtheorem{claim}[thm]{Claim}

\newtheorem{obs}[thm]{Observation}

\crefname{lem}{lemma}{lemmas}

\newcommand{\gO}{\Omega}
\newcommand{\go}{\omega}

\newcommand{\EE}{\mathbb{E}}
\newcommand{\E}{\mathbb{E}}

\newcommand{\cH}{\mathcal{H} }
\newcommand{\h}{\mathcal{H} }

\newcommand{\cR}{\mathcal{R} }
\newcommand{\R}{\mathcal{R} }

\newcommand{\beq}[1]{\begin{equation}\label{#1}}
\newcommand{\enq}[0]{\end{equation}}

\newcommand{\eps}{\varepsilon}

\newcommand{\ga}[0]{\alpha }
\newcommand{\gb}[0]{\beta }
\newcommand{\gc}[0]{\gamma }
\newcommand{\gd}[0]{\delta }
\newcommand{\gD}[0]{\Delta }

\newcommand{\gS}[0]{\Sigma}

\newcommand{\mn}[0]{\medskip\noindent}
\newcommand{\nin}[0]{\noindent}

\newcommand{\sub}[0]{\subseteq}
\newcommand{\sm}[0]{\setminus}

\newcommand{\pr}[0]{\mathbb{P}}

\newcommand{\Aut}[0]{\mbox{Aut}}
\newcommand{\aut}[0]{\mbox{aut}}

\renewcommand{\dots}{,\ldots ,}

\newcommand{\pE}[0]{p_{\mathbb E}}

\newcommand{\C}[2]{\binom{{#1}}{{#2}}}

\newcommand{\ud}{{\underline{d}}}
\newcommand{\ub}{{\underline{b}}}

\newcommand{\dd}{\Delta}

\newcommand{\tdN}{\tilde{N}}
\newcommand{\tdX}{\tilde{X}}

\newcommand{\xypl}[0]{\mbox{$(x,y)$-$P_\ell$}}
\newcommand{\gdd}[0]{\delta}
\newcommand{\dhl}[0]{\tilde{\ell}}

\begin{document}

\title{On the "Second" Kahn--Kalai Conjecture: cliques, cycles, and trees}

\author[Q. Dubroff]{Quentin Dubroff}
\address{Department of Mathematics, Carnegie Mellon University}
\email{qdubroff@andrew.cmu.edu}

\author[J. Kahn]{Jeff Kahn}
\address{Department of Mathematics, Rutgers University}
\email{jkahn@math.rutgers.edu}

\author[J. Park]{Jinyoung Park}
\address{Department of Mathematics, Courant Institute of Mathematical Sciences, New York University}
\email{jinyoungpark@nyu.edu}

\begin{abstract}
We prove a few simple cases of 
a random graph statement that 
would imply the 
"second" Kahn--Kalai Conjecture. 
Even these cases turn out to be reasonably
challenging, and it is hoped that the ideas
introduced here may lead to further interest
in, and further progress on, 
this natural problem. 
\end{abstract}

\maketitle

\section{Introduction}\label{sec.intro}

For graphs $G$ and $J$, a \emph{copy} of $J$ in $G$ is an (unlabeled)
subgraph of $G$ isomorphic to $J$. (We will, a little abusively, use ``$G\supseteq H$"
to mean $G$ contains a copy of $H$.)
We use $N(G,J)$ for the number of such copies, and
$\EE_pX_J$ for $\EE N(G_{n,p},J)$ (where 
$G_{n,p}$ is the usual 
"Erd\H{o}s-R\'enyi" random graph). 
See the end of this section for 
other definitions and notation.

For $q \in [0,1]$, say a graph $J$ is \textit{$q$-sparse} if
\[
\EE_qX_{I} \ge 1 \,\,\,\forall I\sub J.
\]

We are interested here in the following
conjecture from \cite{dubroff2025second}.
\begin{conj}
\cite[Conj.~1.7]{dubroff2025second}
\label{conj:main}
There is a fixed K such that
if $H$ is $q$-sparse and $p=Kq$, then
\[
N(H,F) < \EE_pX_F \quad \forall F \sub H.
\]  
\end{conj}
\nin
(Note ``$\sub H$" is unnecessary.)

This 
simple statement 
is our preferred form of 
\cite[Conj.~1.6]{dubroff2025second},
which would imply the 
"second" Kahn-Kalai
Conjecture \cite[Conj.~2.1]{kahn2007thresholds}. 
We will not go into background here,
just referring to the discussion in \cite{dubroff2025second},
but for minimal context recall
the original conjecture of \cite{kahn2007thresholds},
though it will not be needed below.

Define the \textit{threshold for $H$-containment}, 
$p_c(H)= p_c(n,H)$,
to be the unique $p$ for which
$\pr(G_{n,p}\supseteq H)=1/2$,
and set 
\[
\pE(H) =\pE(n,H)=\min\{p:\EE_pX_I \ge 1/2 ~ \forall I \sub H\}.
\]
This is essentially what 
\cite{kahn2007thresholds} calls the 
\textit{expectation threshold}, though
the name was repurposed in \cite{frankston2021thresholds}.
It is, trivially, a lower bound on
$p_c(H)$ since,
for any $I\sub H$,
$\pr(G_{n,p}\supseteq H) \leq \pr(G_{n,p}\supseteq I) \leq \EE_pX_I$.
The ``second Kahn--Kalai Conjecture" (so 
called in 
\cite{mossel2022second}), 
which was in fact the starting point for 
\cite{kahn2007thresholds}, is then  
\begin{conj}\cite[Conj.~2.1]{kahn2007thresholds}\label{KKC2}
There is a fixed $K$ such that for any graph $H$, 
\[p_c(H) < K \pE(H) \log v_H.\]
\end{conj}
\nin
(That this is implied by \Cref{conj:main}
follows from the main result of 
\cite{frankston2021thresholds}; 
again, see \cite{dubroff2025second}.) 
In the limited setting to which it applies, 
\Cref{KKC2} 
is considerably stronger than the main 
conjecture of \cite{kahn2007thresholds} 
(called the ``Kahn--Kalai Conjecture" in
\cite{talagrand2010many}),
which is now a result of Pham and the third author \cite{park2024proof}. 

At this writing the best we know in the
direction of \Cref{conj:main}
is the main result of
\cite{dubroff2025second}, viz.

\begin{thm}
\cite[Theorem 1.8]{dubroff2025second}
\label{MT}
There is a fixed $K$ such that if $H$ is $q$-sparse and 
$p=Kq\log^2 n$, then
\beq{eq.MT} N(H,F) < \EE_pX_F \quad \forall F \sub H.\enq
Furthermore, there is a fixed $\ga>0$ such that
\eqref{eq.MT} holds if $H$ is $q$-sparse with
$q=\ga p\le 1/(3n)$. 
\end{thm}

In this paper we show that 
\Cref{conj:main} is correct for a few 
simple families of $F$'s, as follows. 
(Note that 
the sizes of these $F$'s can depend on $n$.)

\begin{thm}\label{thm:clique}
    There is a fixed $L$ such that the following holds. Suppose $H$ is $q$-sparse and $p=Lq$. If $F$ is a clique or a cycle, then
\[ N(H,F) < \EE_pX_F.\]
\end{thm}

\begin{thm}\label{thm:tree}
For any $\dd$, there exists an $L=L(\dd)$ such 
that if $H$ is $q$-sparse, $p=Lq$, and $F$ is a tree
with maximum degree $\dd$, then
\[
N(H,F) < \EE_pX_F.
\]
\end{thm}

\nin 
\emph{Remarks.}
(a)  It is easy to see
(see \cite[Proposition 2.4]{dubroff2025second})
that if \Cref{conj:main}
is true for each component of $F$ then
it is true for $F$; 
in particular 
\Cref{thm:tree} implies the 
conjecture for forests as well as trees.

\nin
(b)
Even the above elementary cases are, 
to date, not so easy, and the present
work is meant partly to highlight this,
and partly to give some first ideas
on how to proceed.  
One may of course wonder whether this 
(seeming) difficulty is telling us 
the conjecture is simply wrong, but (and somewhat
contrary to our initial opinion) we now
tend to think it is true.

\mn
\textbf{Outline and preview.}
\Cref{sec.prelim} includes definitions and a few initial observations, following which
the clique portion of \Cref{thm:clique},
\Cref{thm:tree}, and
the cycle portion of \Cref{thm:clique}
are proved in Sections~\ref{sec:clique},
\ref{sec:trees} and \ref{sec:cycles}
respectively.  Of these:

Cliques are our easiest case and may serve as 
a warm-up for what follows. 
\Cref{thm:clique}
for cycles is postponed to 
\Cref{sec:cycles} since it 
depends on the result for paths, a first 
case of \Cref{thm:tree}. 
While
the proof of \Cref{lem:cycle.key}
seems to us quite interesting 
(as does the fact that getting 
from paths to cycles seems not at all 
immediate), 
we regard the proof of \Cref{thm:tree}
as the heart of the paper.
Here it may be helpful to think of the (prototypical) case of paths.
A simpler argument for even this very
simple case would be welcome, as
(of course) would 
be a proof of \Cref{thm:tree}
without the degree restriction.

\subsection*{Usage} 
For a graph $J$ we use
$v_J$ and $e_J$
for $|V(J)|$ and $|E(J)|$,
and $\gD_J$ for the maximum degree in $J$. 
The identity of $H$
(in Theorems~\ref{thm:clique} 
and \ref{thm:tree}) is fixed
throughout, and we often 
use \emph{copy of $J$} for
\emph{copy of $J$ in H}.

As usual,
$J[U]$ is the subgraph of $J$ 
induced by $U\sub V(J)$, and 
$v\sim w$ denotes adjacency of $v,w\in V(J)$.
For $A, B \sub V(J)$
(here always disjoint), 
$\nabla_J(A, B):=\{\{v,w\} \in E(J):v \in A, w \in B\}$ and $\nabla_J(A):=\nabla_J(A, V(J) \setminus A)$. We also use $\nabla_J(v)$ for $\nabla_J(\{v\})$ 
(and similarly for $\nabla_J(v,\cdot)$) and $d_J(v)=|\nabla_J(v)|$.

Recall (see e.g. \cite{janson2011random}) that the \emph{density} of a graph $J$ with $v_J \neq 0$ is 
$d(J)=e_J/v_J$, and the 
\textit{maximum density} of $J$ is 
$m(J) =\max\{d(I):I \sub J\}$.

Throughout the paper, $\log$ means $\log_2$. 
For positive integers $a$ and $b$, we use
$[a]=\{1,2, \ldots a\}$, $[a,b]=\{a, a+1, \ldots, b\}$, and $(a)_b=a(a-1)\cdots(a-b+1)$.  
We make no effort to keep our constant
factors small, and,
in line with common practice, often
pretend large numbers are integers.

\section{Preliminaries}\label{sec.prelim}

Note that, in proving \Cref{conj:main}, we may 
assume $n$ is somewhat large, since otherwise the conjecture 
is vacuous for large enough $L$. We may also 
assume that $L$ is somewhat large, so
\beq{qL}
q =p/L \le 1/L
\enq
is somewhat small.

We will make occasional, usually tacit, use
of the familiar fact that
for positive integers $a,b$, 
\beq{elementary}
(a)_b > (a/e)^b. 
\enq

\begin{prop}\label{prop:max.deg}
If $H$ is $q$-sparse, then $\gD_H\le \max\{\log n, 2enq\}.$ In particular, if $q\ge \log n/n$, then $\gD_H\le 2enq$.
\end{prop}

\begin{proof}
If $R$ is a $k$-star with $k> \max\{\log n, 2enq\}$,
then (using \eqref{elementary} for the second inequality)
\[
\EE_qX_R
< n \C{n}{k}q^k
< n\left(\frac{enq}{k}\right)^k< n 2^{-k}<1,
\]
so $R\not\sub H$.
\end{proof}

\begin{prop}\label{prop:dense}
If $H$ is $q$-sparse, then $m(H)<\log n$. If in addition 
$q \le  n^{-c}$, then $m(H)< 1/c$.
\end{prop}

\begin{proof}
If $d(R)\ge \log n$ (that is, $e_R \ge v_R\log n$), then
\[
\EE_qX_R
< n^{v_R}q^{e_R}\
\le \left(nq^{\log n}\right)^{v_R}
\le (nL^{-\log n})^{v_R}<1
\]
(see \eqref{qL}); so $H\not\supseteq R$.

Similarly, if $q\le n^{-c}$ and 
$d(R) \geq 1/c$, 
then 
\[
\EE_qX_R < n^{v_R}q^{e_R}\le
\left(nq^{1/c}\right)^{v_R}\le 1
\]
(so $H\not\supseteq R$).
\end{proof}

\begin{cor}\label{cor:e_H}
If $H$ is $q$-sparse, then $e_H< n\log n$,
and $e_H < n/c\,$ if $\,q \le n^{-c}$.
\end{cor}

We denote by $\nu(H,J)$ the maximum size of an 
edge-disjoint collection of copies of $J$ in $H$. 
The following simple observation will be important.
\begin{prop}\label{prop:nu}
If $H$ is $q$-sparse, then for any $J$, $\nu(H,J) \le e\EE_qX_J.$
\end{prop}

\nin
This is helpful because (roughly):  trivially, 
\beq{N.nu.B}
N(H,J)\le \nu(H,J)\cdot B
\enq
for any bound $B$
on the number of copies of $J$ (in $H$) 
sharing an edge with 
a given copy; and possible bounds $B$ should be 
better than bounds on $N(H,J)$ itself, since
the number of starting points
for a copy of $J$ meeting a given copy is 
at most $v_J$, rather than
the usually much larger $n$.
This idea plays a main role below, and again
in \cite{dubroff2025second},
which was inspired in large part by the ideas
introduced here.

\begin{proof}
Let $R$ be 
the edge-disjoint union of $\nu$ copies of $J$,
with $\nu>e\EE_qX_J$. Then
\[
\EE_qX_R 
=N(K_n,R)q^{e_R}\le {N(K_n,J) \choose \nu}
q^{\nu\cdot e_J}< \left(\frac{eN(K_n,J)q^{e_J}}{\nu}\right)^{\nu}=\left(\frac{e\EE_qX_J}{\nu}\right)^\nu < 1,
\]
so $H\not\supseteq R$.
\end{proof}

\section{Cliques}
\label{sec:clique}

Here we prove the clique portion of \Cref{thm:clique}.
Recall that $p=Lq$, with $L$ fixed and somewhat 
large (large enough to support the assertions below),
and let $F=K_{r+1}$ ($r \in [2,n-1]$) (noting that $r =1$, which could easily
be included here, is immediate from \Cref{prop:nu}). We 
divide possibilities for $q$ into two ranges, for which 
we use different arguments.

\mn
\emph{Small $q$}.
Suppose
\beq{eq:p.ub} q< n^{-2/(r+1)}.\enq

\nin
This is our first use of the strategy sketched
following \Cref{prop:nu}; the desired bound
on $B$ 
(in \eqref{N.nu.B}) 
is provided by the next observation.

\begin{lem}\label{cl:no.three.K'}
If $H$ is $q$-sparse (with $q$ as in \eqref{eq:p.ub})
and $K$ is a copy of $F$ (in $H$),
then number of copies of $F$ that share edges with $K$ is less than $(er)^{r+1}$.
\end{lem}

\begin{proof}
Let $R$ be the union of the copies of $F$ that share  edges with $K$. We show that $v_R$ can't be 
too large, and, given this,
use the crudest possible
bound on $N(R,F)$.

Set $K=R_0$ and choose copies 
$R_1, R_2\dots R_m$ of $F$ that
share edges with $K$ and
satisfy
\[
E(R_i) \not\sub \cup_{j<i} E(R_j) 
\,\,\forall i \in [m]\,\,\,\, \text{and} \,\,\,\,\cup_{i=0}^m R_i=R.
\]
We claim that 
\beq{mler}
m\le r.
\enq

\begin{subproof} Set 
$v_i=|V(R_i)\sm \cup_{j <i} V(R_j)|$ and $e_i=|E(R_i) \setminus \cup_{j <i} E(R_j)|$.  Then
(since $H$ is $q$-sparse, and using 
\eqref{eq:p.ub} for the third inequality)
\beq{eq:expR}
1\le \EE_qX_R<  
n^{v_R}q^{e_R}= n^{r+1}q^{\C{r+1}{2}}\prod_{i=1}^m \left(n^{v_i}q^{e_i}\right) <
n\prod_{i=1}^m \left(n^{v_i}q^{e_i}\right).
\enq
Again by \eqref{eq:p.ub}, we have
$n^{v_i}q^{e_i} < n^{-2/(r+1)}$ if $v_i = 0$, while
$v_i \in [r-1]$ gives 
$e_i \ge {{r+1} \choose 2}-{r+1 -v_i \choose 2}$ and 
\[
n^{v_i}q^{e_i} < n^{v_i -2\left({{r+1} \choose 2} - {r+1 -v_i \choose 2}\right)/(r+1)}.
\]
Here the exponent on the r.h.s.\ 
is maximized (over $v_i \in [r-1]$)
at $v_i = 1$ and $v_i = r-1$, yielding
\[
n^{v_i}q^{e_i} < n^{-1 + 2/(r+1)}.
\]
So in any case, 
$n^{v_i}q^{e_i} < n^{-1/(r+1)}$
(since $r \geq 2$),
and the r.h.s.\ of \eqref{eq:expR} is less than 
$   
n \cdot n^{-m/(r+1)},
$   
yielding \eqref{mler}.\end{subproof}

Thus $v_R \leq r^2+1$ (say) and
$   
N(R,F) \leq \C{r^2+1}{r+1}< (er)^{r+1}. 
$     
\end{proof}

Finally, the combination of \Cref{prop:nu} and \Cref{cl:no.three.K'} gives (for slightly large $L$)
\[
N(H,F)\le \nu(H,F) \cdot (er)^{r+1} \le 
e\cdot \EE_qX_F(er)^{r+1}<L^{{r+1} \choose 2}\EE_qX_F= \EE_pX_F,
\]
so we have \Cref{thm:clique} in this case.\qed

\mn
\emph{Large q}.
Now suppose
\[   
q \ge n^{-2/(r+1)}.
\]   
Then 
\beq{eq:lb1} \EE_pX_F= \C{n}{r+1}p^{\C{r+1}{2}}
\ge \left(\frac{np^{r/2}}{r+1}\right)^{r+1} \ge n\left(\frac{L^{r/2}}{r+1}\right)^{r+1}>n\cdot (L/4)^{r(r+1)/2}.\enq
We will argue by contradiction, showing that if
$N(H,F) \ge \EE_pX_F$, then
there is an $R \sub H$ with $\EE_qX_R<1$.

Let 
\beq{eq:a.def} 
a=\EE_pX_F/n~\left(>(L/4)^{r(r+1)/2}\right);
\enq
so we are assuming
\[N(H,F)\ge an.\]

Recall that a \emph{hypergraph}, $\h$,
on (\emph{vertex set}) $V$ is a collection of 
subsets (\emph{edges}) of $V$;
\emph{degree} for hypergraphs
is defined as for graphs.
We recall a standard fact:

\begin{obs}\label{obs:min.deg} Let $\cH$ be the hypergraph on $V(H)$ whose edges are the 
vertex sets of 
copies of $F$ in $H$. With $a$ as in \eqref{eq:a.def}, there is a $W \sub V(H)$ such that
\beq{eq:min.deg}
\text{$\cH[W]$ has minimum degree at least $a$.}
\enq
\end{obs}

\begin{proof}
Set $\h_0=\h$ and for $i\geq 1$ until no longer 
possible, let $\h_i$ be gotten from $\h_{i-1}$
by removing a vertex of degree 
less than $a$ 
(and the edges containing it).  
The final hypergraph is nonempty
(since we delete fewer than $an\le N(H,F)=|\h|$ 
edges) and has minimum degree at least $a$.   
\end{proof}

Fix $W$ as in \Cref{obs:min.deg} and set $R=H[W]$;
so each vertex of $R$ is contained in at least $a$ copies of $F$ in $R$. Write $\gd$ for the minimum degree in $R$ and $w$ for $|W|$ ($=v_R$). Then
\[
\EE_qX_R< n^wq^{e_R}\le \left(nq^{\gd/2}\right)^w,
\]
so we will have the desired contradiction 
$\EE_qX_R <1$ if we show 
\[q^{\gd/2}<1/n.\]
To this end, we find a suitable lower bound on $\gd$ and upper bound on $q$. 

For the first of these, 
our choice of $W$ and definition of $\gd$ give 
$a \le {\gd \choose r}< \left(\frac{e\gd}{r}\right)^r$,
so
\beq{gd.lb}\gd > ra^{1/r}/e.\enq

For an upper bound on $q$, in view of
\eqref{eq:a.def} and \eqref{eq:lb1}, we have
$ an=\EE_pX_F > \left(\frac{np^{r/2}}{r+1}\right)^{r+1}> \left(nq^{r/2}\right)^{r+1}$
(provided $L^{r/2}> r+1$), whence
\beq{q.ub} q<(a^{1/r}/n)^{2/(r+1)}.\enq

Since $R\sub H$, \Cref{prop:dense} promises
$\gd < 2\log n$, which with \eqref{gd.lb}
gives 
\beq{eq:a.ub} a^{1/r}< 2e\log n/r;\enq
and inserting this in \eqref{q.ub} 
(and again using \eqref{gd.lb}) we have
(with room)
\[
q^{\gd/2}<\left(\frac{2e\log n}{rn}\right)^{\gd/(r+1)}< \left(\frac{2e\log n}{rn}\right)^{a^{1/r}/(2e)} < 1/n,
\]
where the last inequality uses 
$a >(L/4)^{r(r+1)/2}$
(see \eqref{eq:a.def}). 

This completes the proof of 
\Cref{thm:clique} for cliques.

\section{Trees}\label{sec:trees}

Here we prove \Cref{thm:tree}.
We now use $\eps$ for what will be 
$1/L$; so $\eps$ 
\emph{is a function of $\gD=\gD_F$} and
$q= \eps p$.
We assume $\eps$ is small enough to support 
what we do, and, as usual, don't 
try to give 
it a good value.

Let $F$ be a tree, say with $e_F=j\,$ 
($\in [n-1]$), and set
\[
d=np.
\]

It will be convenient to work with \emph{labeled}
copies (a labeled copy of $J$ in $G$ being an
injection from $V(J)$ to $V(G)$ that takes edges 
to edges).
We use $\tdN(G,J)$ for the number of of labeled copies
of $J$ in $G$ and
$\EE_p\tdX_J$ for $\EE \tdN(G_{n,p},J)$. 
Then 
$N(H,F)=\tdN(H,F) / \aut(F)$ and $\EE_p X_F = \EE_p\tdX_F/\aut(F)$
(where $\aut(\cdot):=|\Aut(\cdot)|$), and 
the inequality of \Cref{thm:tree}
is the same as $\tdN(H,F) < \EE_p \tdX_F$;
so, since
\beq{EEptd}
\EE_p\tdX_F=(n)_{j+1}p^j > e^{-(j+1)}nd^j
\enq
(see \eqref{elementary}),
the theorem will follow from
\beq{goal:tree}
\tdN(H,F) \le \eps^{0.1j}nd^j
\enq
(provided $\eps < e^{-20}$),
which is what we will prove.

\subsection{Set-up and definitions.} Let $V(F)=\{v_0, \ldots, v_j\}$, where we think of $F$ rooted at $v_0$
and $(v_1\dots v_j)$ is some breadth-first order. 
Let $f_i$ be the number of children of $v_i$ (so $f_i $ is $ d_F(v_i)$ if $i=0$ 
and $ d_F(v_i) - 1$ otherwise).

Before turning to the main line of argument, 
we dispose of two easy cases.
\begin{prop}\label{cl:d.range}
The inequality in \eqref{goal:tree} holds if 
$d \le 1/(3\eps)$ or $ d \ge\eps^{-1/3}\log n$.
\end{prop}
\begin{proof}
The assertion for $d \le 1/(3\eps)$
follows from (the second part of) \Cref{MT}.
(In more detail:  
assume $\eps < \ga^{10/9}$ with $\ga$ as
in the theorem, and let $p'=q/\ga$; so 
$\ga p'=q \le 1/(3n)$ 
and $p' < \eps^{0.1}p$,
implying $N(H,F)< \E_{p'}X_F
< \eps^{0.1j}\E_pX_F$.)

If, on the other hand, 
$d \ge \eps^{-1/3}\log n$, then 
\Cref{prop:max.deg} gives
\[
\gD_H \le \max\{\log n, 2enq\} \le \max\{\eps^{1/3}d,2e\eps d\}=\eps^{1/3}d ,
\]
whence
\[
\tdN(H,F)
\le 2e_H \gD_H^{j-1} \le 2n\log n \cdot (\eps^{1/3}d)^{j-1} \le 2\eps^{j/3}nd^j
\]
(where 
$2e_H$ bounds the number of embeddings of $v_0v_1$,
$\gD_H^{j-1}$ bounds the number of ways to 
extend to the rest of $F$, 
and the second inequality uses \Cref{cor:e_H});
so we have \eqref{goal:tree}. 
\end{proof}

So we assume from now on that
\beq{WMAd}
1/(3\eps) < d <\eps^{-1/3}\log n.
\enq

\nin
\emph{Remark.}
With small modifications, the following argument
goes through without the lower bound in \eqref{WMAd}, 
and, in cases where $q < 1/n$, 
without the bounded degree assumption in \Cref{thm:tree}.
In particular,
since for $q<1/n$ a $q$-robust $H$ is acyclic,
this gives an alternate 
proof of a \emph{slight} strengthening of
the second part of \Cref{MT} 
(namely, replacing $q<1/(3n)$ by $q<1/n$),
which, strangely, we don't see how to
squeeze out of the argument in
\cite{dubroff2025second}.

\begin{mydef}[Legal degree sequence]
Say $\underline d=(d_0, \ldots, d_{j})$ 
is \textit{legal} if for all 
$i \in [0, j]$,
\[
\mbox{either 
$d_i \ge \sqrt \eps d$ ($i$ is \textit{big}) or $d_i=f_i$ ($i$ is \textit{small}).}
\]
\end{mydef}
Note that
\beq{figD}
f_i \leq \gD < \sqrt \eps d \,\,\,\forall i,
\enq
since \eqref{WMAd} gives
$\sqrt \eps d > 1/(3 \sqrt \eps)$, 
which we may assume is
greater than $\gD$; 
so ``big" and ``small" do not overlap.
\emph{From now on $\ud$ is always a legal degree sequence.}

\begin{mydef}[Partially labeled $R$]
\label{Rdef}
For a legal $\underline d$, define $\cR_{\underline{d}}$ to be the set of partially labeled graphs $R$ that consist of 
\begin{enumerate}[(i)]
    \item $F$ (with its labels) plus
    \item for each $i \in [0, j]$, $d_i$ edges joining $v_i$ to vertices not in $\{v_0, \ldots, v_{i-1}\}$
    (which may still be in $F$ but should be thought of as mostly new); vertices of $R\sm F$ are unlabeled.
\end{enumerate}
\nin
A \emph{copy} (in $H$) of such an $R$ is then partially labeled in the same way.
\end{mydef}

Set $\cR=\cup \cR_{\underline{d}}$. 
We use $\hat R$ for a copy of $R$, 
$\hat F$ for a (\emph{labeled}) copy of $F$,  
$\hat \cR_{\underline{d}}$ for the set of 
copies of $R$'s in 
$\cR_{\ud}$,
and $\hat{\cR}=\cup \hat{\cR}_{\ud}$.
We write 
$\hat R \sim \hat F$ if $\hat F$ is the "$F$-part" of $\hat R$.

\begin{mydef}[Fit]
For $\hat R \sub H$ a copy of $R \in \cR_\ud$, with $w_i \in V(\hat R)$ the copy of $v_i$, say 
$\hat R$ \emph{fits} $H$ if, for all 
$i \in [0,j]$, 
\beq{fit}
|N_H(w_i) \setminus \{w_0, \ldots, w_{i-1}\}|\,\,
\left\{\begin{array}{cl}
=d_i  & \text{if $i$ is big,}\\
 < \sqrt \eps d & \text{if $i$ is small.}
\end{array}\right.
\enq
\end{mydef}

\begin{obs}\label{obs.fit}
For each labeled $\hat F \sub H$, there is
a unique $\hat R \in \cR$ such that 
$\hat R \sim \hat F$ and $\hat R$ fits $H$.
\end{obs}
\nin
(With $w_i$ the copy of $v_i$ in $\hat F$, 
the desired $\hat R$ consists of $\hat F$ plus
all edges $w_iu$ with 
$u\in N_H(w_i) \setminus \{w_0, \ldots, w_{i-1}\}$
and 
$|N_H(w_i) \setminus \{w_0, \ldots, w_{i-1}\}|\ge 
\sqrt \eps d$ 
(and the vertices in these edges).)

\mn

For $R \in \cR$, let $N^*(H, R)$ 
be the number of 
copies of $R$ that fit $H$. By \Cref{obs.fit},
\beq{N*bd}
\tdN(H,F) = \sum_{R \in \cR} N^*(H, R).
\enq

\nin
\textbf{Plan.}
We will give two upper bounds on 
$\sum_{R \in \cR_{\underline{d}}} N^*(H,R)$
and show that, for each 
$\ud$, 
one of these is small.  Which bound we use will depend
on how 
\beq{D(d)}
D(\ud):=\sum_{\text{$i$ big}} d_i
\enq
compares
to $j\log d$, but in either case will be small enough relative to the bound of \eqref{goal:tree} that even summing
over $\ud$ causes no trouble.

We conclude this section by 
showing that the cost of "decomposing" $D$ is 
small.

\begin{prop}\label{partition.cost} 
For any D 
the number of $\ud$'s with $D(\ud)=D$ is 
\beq{count.bd}
\begin{array}{ll}
\exp\left[O\left(j\log^2 d/(\sqrt{\eps} d)
\right)\right] & \text{ if } D \leq j \log d,\\
\exp\left[O\left(D\log d/(\sqrt{\eps} d)
\right)\right] & \text{ if } D > j \log d.
\end{array}
\enq
\end{prop}

\begin{proof} 
The number of big $i$'s for a $\ud$
with $D(\ud)=D$ is at most
\[
s_0 := \min\{j, D/(\sqrt \eps d)\} < 
\sqrt{jD/2}
\]

\nin
(see \eqref{WMAd}),
and the number of such $\ud$'s 
with exactly $s$ big $i$'s is less than 
\beq{si's}
\C{j}{s}\C{D-1}{s-1} < \exp_2[s\log (e^2jD/s^2)];
\enq
so (since the r.h.s.\ of \eqref{si's}
increases rapidly with $s$), 
the number of $\ud$'s in the proposition is less than
\beq{0th.bd}
\sum_{s \le s_0}
\exp_2[s\log (e^2jD/s^2)]
< 2
\exp_2\left[s_0\log (e^2jD/s_0^2)\right],
\enq
which is
\begin{eqnarray}
2 \exp_2\left[D/(\sqrt{\eps} d)\log (e^2j\eps d^2/D)\right] 
& \text{if $j > D/(\sqrt{\eps}d)$,}
\label{1st.bd}\\
2\exp_2\left[j\log (e^2D/j)\right] & \text{if $j \leq D/(\sqrt{\eps} d)$}.
\label{2nd.bd}
\end{eqnarray}

Now for \eqref{count.bd}: 
If $D \leq j \log d$, then \eqref{1st.bd}
applies (since $d$ is somewhat large; see \eqref{WMAd}), so the bound in 
\eqref{0th.bd} is at most
\[
2 \exp_2\left[j \log d/(\sqrt{\eps} d)\log (e^2\eps d^2/\log d)\right] = \exp\left[O\left(j \log^2 d/(\sqrt{\eps} d)\right)\right]
\]
(using the fact that $x\log(\ga/x)$ is 
increasing in $x$ up to $\ga/e$).
And if $D > j \log d$, then:
if $j \leq D/(\sqrt \eps d)$ then
the version \eqref{2nd.bd}
of the bound in \eqref{0th.bd} is at most
\[
\exp\left[O\left(D/(\sqrt \eps d)\log (e^2 \sqrt \eps d)\right)\right] = \exp\left[O\left(D\log d/(\sqrt{\eps} d)\right)\right];
\]
and otherwise we use
$j < D/\log d$ to say the bound in
\eqref{1st.bd} is less than
\[
2 \exp_2\left[D/(\sqrt{\eps} d)\log (e^2\eps d^2/\log d)\right]
= \exp\left[O\left(D\log d/(\sqrt{\eps} d)\right)\right].\qedhere
\]
\end{proof}

\subsection{First bound.}  The goal of this section is to show 
\beq{1st.goal} 
\sum_{D(\ud) \le j\log d}\,\sum_{R \in \cR_\ud}N^*(H,R)
< n (\eps^{1/3} d)^j.
\enq
We first bound the inner sums and then
invoke \Cref{partition.cost}.

\begin{prop}\label{cl:bd1} For any $\ud$,
\beq{first.bd} 
\sum_{R \in \cR_\ud}N^*(H,R) \le    
n \prod_{\text{$i$ small}}(\sqrt \eps d)^{f_i}
\prod_{\text{$i$ big}} d_i^{f_i}.
\enq
\end{prop}

\begin{proof}
This is just the naive bound on the number of 
$\hat F$'s for which the unique 
$\hat R\sim\hat F$ that fits $H$
(see \Cref{obs.fit}) is in $\R_\ud$.
With $w_i$ again the copy of $v_i$ in $\hat F$,
we choose $w_0\dots w_j$ in turn.
The number of choices for $w_0 $  is 
at most $n$, and, since 
$\hat R$ is in $ \R_\ud$ and fits $H$, 
the number of choices for the children of $w_i$  
(which are all chosen with 
$(w_0\dots w_i)$ known) is
at most 
\[
(|N_H(w_i) \sm \{w_0, \ldots, w_{i-1}\}|)_{f_i},
\]
which with \eqref{fit} gives \eqref{first.bd}.
\end{proof}

\begin{prop}\label{two.products}
If 
\beq{D.ub}
D:=
D(\ud) \le j\log d,
\enq
then
\beq{prodprod}
\prod_{\text{$i$ small}}(\sqrt \eps d)^{f_i}\prod_{\text{$i$ big}}d_i^{f_i} \le 
(\sqrt \eps d)^j
e^{(\dd j\log^2 d)/(\sqrt \eps d)}.
\enq
\end{prop}

\begin{proof} 
Since $\sum f_i=j$, 
the first product is less than
$(\sqrt \eps d)^j$.
For the second, with $s$ the number of 
big $i$'s, we have
$s \le D/(\sqrt \eps d) ~(<D/e),$
whence (using \eqref{D.ub} 
for the last inequality)
\[
\prod_{\text{$i$ big}}d_i^{f_i}\le
\prod_{\text{$i$ big}} d_i^{\dd} \le (D/s)^{s\dd}
\le (\sqrt \eps d)^{\dd D/(\sqrt \eps d)}\le
e^{(\dd j\log^2 d)/(\sqrt \eps d)}.\qedhere
\]
\end{proof}

Finally, inserting \eqref{prodprod}
in \eqref{first.bd} and using
\Cref{partition.cost},
we find that the l.h.s.\ of \eqref{1st.goal} 
is at most 
\[
\left\{(j\log d)\cdot
\exp\left[(\gD+O(1))j\log^2 d/(\sqrt \eps d)\right]
\right\}\cdot n\cdot  ( \sqrt{\eps} d)^j 
< n ( \eps^{1/3} d)^j
\]
Here the inequality holds because, since
$d > 1/(3\eps)$ (see \eqref{WMAd}), 
the expression in $\{\,\}$'s is much smaller
than $\eps^{-j/6}$ for a small enough 
$\eps $ ($=\eps(\gD)$).

\subsection{Second bound.}   
Here we have more room and will show 
\beq{2nd.goal} 
\sum_{D(\ud) > j\log d}\,\sum_{R \in \cR_\ud}N^*(H,R) \le  n \eps^{(j/3)\log d}. 
\enq
\nin
(What we say here
applies to \emph{any} $\ud$
until we get to the end of the section,
where we finally use $D(\ud) > j\log d$.)

For this discussion 
$R$ is always in some
$\R_\ud$, so, as in \Cref{Rdef},
copies of $R$ are partially labeled; with this understanding, we 
again use $N(G,R)$ for the number of copies of
$R$ in $G$, $\E_pX_R$ for $\E N(G_{n,p},R)$, and 
$\nu(H,R)$ for 
the maximum size of an 
edge-disjoint collection of copies of $R$ in $H$.

Like the treatment of small $q$ in 
\Cref{sec:clique}, the proof of 
\eqref{2nd.goal} 
uses the approach previewed following 
\Cref{prop:nu}; thus we hope for a bound
on the inner sum in \eqref{2nd.goal}
of the form
\beq{form}
\sum_{R \in \cR_\ud} \nu(H,R) \cdot B(R),
\enq
where $B(R)$ is some bound on the number of copies
of $R$ that fit $H$ and share edges with a given copy.
(Here:  (i) we would be entitled to insist that in
the definition of $\nu(H,R)$ we restrict to copies of
$R$ that fit $H$, but we don't need---and anyway don't know how to 
use---this; (ii) rather 
than $B(R)$, we will use a single bound, $\gb(\ud)$, 
on the number of \emph{all} copies of 
$R$'s in $\R_\ud$ 
that fit $H$ and share edges with a given 
copy---though a bound on the number of such
copies of a \emph{single} $R$ could in principle 
be much smaller.)

\mn

We observe that \Cref{prop:nu}
trivially 
(and with some sacrifice) extends to copies of
$R$:
\begin{cor}\label{cor.nu}
If $H$ is $q$-sparse, then for any $R\in \R$, 
$\nu(H,R) \le e\EE_qX_R.$
\end{cor}
\begin{proof}
With $S$ the unlabeled graph underlying $R$, 
we have (using \Cref{prop:nu})
\[
\nu(H,R)=\nu(H,S) \le e\E_qX_S\leq e\E_qX_R.\qedhere
\]
\end{proof}

\begin{lem} \label{Lnubd}
For any $\ud$,
\beq{bd1.1'}
\sum_{R \in \cR_\ud} \nu(H,R) \,<\,
e n \prod_{\text{$i$ small}} (\eps d)^{f_i} \prod_{i \text{ big}} \left[\left(\frac{e\eps d}{d_i}\right)^{d_i} d_i^{f_i}\right]\,=:\,\ga(\ud).
\enq
\end{lem}

\begin{proof}
By \Cref{cor.nu}
it is enough to show
\beq{bd1.1} 
\sum_{R \in \cR_\ud} \EE_qX_R <\ga(\ud)/e.
\enq
Here the main point is to show
\beq{sumN}
\sum_{R \in \cR_\ud} N(K_n,R) <
n \prod_{\text{$i$ small}} n^{f_i}
\prod_{i \text{ big}} \C{n}{d_i}d_i^{f_i}, 
\enq
which,
since $\sum d_i =e_R$ for any $R\in \R_\ud$,
implies that the l.h.s.\ of \eqref{bd1.1} is less than
\[
n \prod_{\text{$i$ small}} (nq)^{f_i}
\prod_{i \text{ big}} \C{n}{d_i}d_i^{f_i}q^{d_i}
\le
n \prod_{\text{$i$ small}} (\eps d)^{f_i} 
\prod_{i \text{ big}} \left[\left(\frac{e\eps d}{d_i}\right)^{d_i} d_i^{f_i}\right].
\]

For the proof of \eqref{sumN}
we continue to use $w_i$ for the copy of 
$v_i$ in $\hat F$,
and now write $p(w_i)$ for the parent of $w_i$. 
We think of 
choosing $w_0$ (in at most $n$ ways) and then 
"processing" (in order)
$w_0\dots w_j$.

If $i$ is small, then "processing" $w_i$ 
means choosing its $f_i$ (labeled) children,
the number of possibilities for which is 
less than $(n)_{f_i}\leq n^{f_i}$.

If $i$ is big, then "processing" $w_i$ means 
choosing the \emph{set} of its $d_i$ neighbors 
not in $\{w_0\dots w_{i-1}\}$, and 
its children (with labels) from this set.
The number of ways to do this
(not all of which will  
lead to legitimate $\hat R$'s) 
is 
at most
\beq{naive}
\C{n}{d_i}d_i^{f_i}.
\enq
\end{proof}

We next bound the number of members of $\hat \cR_\ud$ that fit $H$ and share edges with a given copy 
$\hat R_0$.  
We first slightly refine $\hat{\cR}_\ud$.
For $ \hat{R} \in  \hat{\cR}$, set 
$\ub=\ub(\hat{R})=(b_1, b_2, \ldots, b_j)$, with
\[
b_i =b_i(\hat{R})=
|N_H(w_i) \cap \left(\{w_0, \ldots, w_{i-1}\} \setminus \{p(w_i)\}\right)|
\]
(recall $p(w_i)$ is the parent of $w_i$)
and define $\hat{\cR}_{\ud,\ub}$ 
in the natural way. 
The next observation will allow us to more or less 
ignore edges of $\hat{R}\sm \hat{F}$ 
with both ends in $\hat{F}$.

\begin{prop}\label{bi.sum} 
If $\hat{R}\sub H$, with 
$\ub(\hat{R}) = \ub$, then
\beq{bi.ub} 
\sum_{i \in [j]} b_i \le \max\{1, 3j\log\log n/\log n\}
=\max\{1,o(j)\}.
\enq
\end{prop}

\begin{proof}
Set $\sum b_i = \gd j$ and $S=H[V(\hat{F})]$. Then $v_S=j+1$, $e_S= (1+\gd)j,$ and
\beq{EqXS}
\EE_qX_S\le n^{j+1}q^{(1+\gd)j}= n^{j+1}(\eps d/n)^{(1+\gd)j}=n^{1-\gd j}(\eps d)^{(1+\gd)j}\le n^{1-\gd j} (\log n)^{(1+\gd)j},
\enq
where the last inequality uses \eqref{WMAd}
(weakly) to say $\eps d <\log n$.

If $\gd j \ge 2$ (i.e.\ $\sum b_i>1$), 
then the r.h.s.\ of \eqref{EqXS}
is at most $(\log^{1+\gd} n/n^{\gd/2})^j$, which is less than 1
(in fact $o(1)$)
if $\gd >3 \log\log n/\log n$;
and \eqref{bi.ub} follows since $H$ is $q$-sparse. 
\end{proof}

\begin{prop}\label{prop:b.bd}
The number of possibilities for $\ub$ is at most 
$j+ e^{o(j)}$.
\end{prop}
\begin{proof}
Let $B=\sum_{i \in [j]} b_i$; so  
\Cref{bi.sum} says
$B \le \max\{1,o(j)\}$.
Given $B$, the number of possibilities for $\ub$
is at most $\C{B +j-1}{B}$, which is $j$ if $B=1$ and 
$e^{o(j)}$ if $B=o(j)$; so, crudely, the 
number of possible $\ub$'s is at most 
$j+o(j)e^{o(j)}=j+e^{o(j)}$.
\end{proof}

If $\hat R\in \hat \R_{\ud,\ub}$ fits $H$, then (for any $i$) 
\beq{H.degree}
d_H(w_i)\le \max\{d_i,\sqrt{\eps}d\}+b_i+1_{\{i\neq 0\}}
\,=:\, B_i,
\enq
where we note that the max is $d_i$ if $i$ is big
(in which case \eqref{H.degree} holds with equality),
and $\sqrt{\eps}d$ if $i$ is small (in which case
\eqref{H.degree} is strict).

For $\ell \in [j]$, let
\[
Q(\ell)=\{i \in [j]: \text{$v_i$ is an internal 
vertex of the path in $F$ connecting $v_0$ and $v_\ell$}\}
\]
(that is, $Q(\ell)$ is the set of indices of 
non-root ancestors of $v_\ell $).

\begin{lem}\label{beta.def} 
For any $e \in H$, $\ud$ and $\ub$, 
the number of $\hat R \in \hat \cR_{\ud, \ub}$ that contain $e$ and fit $H$ is at most
\beq{cases.intersection} 
2\sum_{\ell=0}^{j}\prod_{\text{$i$ small}} (\sqrt \eps d)^{f_i} \prod_{\text{$i$ big}} d_i^{f_i} \cdot K(\ell),
\enq
where 
\beq{Kell}
K(\ell)\,
= \,
\prod_{\substack{\text{$i$ big} \\ 
i \in Q(\ell)}} \left(\frac{d_i+b_i+1}{d_i}\right){\prod_{\substack{\text{$i$ small}\\
i \in Q(\ell) }}}\left(\frac{\sqrt \eps d+b_i+1}{\sqrt \eps d}\right)
\cdot \frac{B_\ell}{\max\{d_0,\sqrt{\eps}d\}}.
\enq
\end{lem}

\begin{proof}
We first choose an end, $w$, of $ e$ in $V(\hat F)$ 
(where $\hat R\sim \hat F$; 
this gives the 2 in \eqref{cases.intersection}), 
and the role, $w_\ell$, 
of $w$ in $\hat F.$ 
It is then enough to bound the number of possibilities
for the rest of $\hat R$ by the $\ell^{\text{th}}$ summand in \eqref{cases.intersection}.

Note that for $\ell=0$ (where $K(\ell)=1$), 
the summand is just the bound of
\Cref{cl:bd1}, except that
we no longer need the factor $n$ since we 
already know $w_0$.

For a general $\ell$ 
we first specify $(w_i:i\in Q(\ell)\cup \{0\})$, 
the number of possibilities for which is, by \eqref{H.degree}, at most 
\beq{prodB}
\prod_{i\in Q(\ell)\cup \{\ell\}}B_i.
\enq
Then, for the number of ways to choose the
rest of $\hat R$, we
again argue as in \Cref{cl:bd1}, now
skipping terms in the bound corresponding
to choosing the already known $w_i$'s with
$i\in Q(\ell)\cup \{0, \ell\}$;
this bounds the number of possibilities
by the double product in 
\eqref{cases.intersection} divided by
\[
\prod_{i\in Q(\ell)\cup \{0\}}
\max\{d_i,\sqrt{\eps}d\},
\]
and multiplying by \eqref{prodB}
gives the promised $\ell^{\text{th}}$ summand. 
\end{proof}

From now until the last paragraph
of this section, we fix $\ud$
and let 
$D=D(\ud)$ ($:=\sum_{\text{$i$ big}} d_i$;
see \eqref{D(d)}). 
We have included the $K(\ell)$'s in 
\Cref{beta.def} to help keep track of
what the proof is doing, 
but 
will use only the simplifying 
\beq{d.b.bd}
K(\ell)\le K:= 
(D+j)\cdot\prod_{i \in [j]}\left(1+(b_i+1)/\sqrt \eps d\right) 
< (D+j)\exp\left[(j+\sum_{i\in [j]}b_i)/\sqrt \eps d\right] 
<(D+j)\exp[O(j)/\sqrt \eps d],
\enq
where 
$D+j$
corresponds to the trivial $B_\ell \leq D+j$, 
and the last inequality 
uses \Cref{bi.sum}
(and the $O(j)$ in the final exponent
is actually $(1+o(1))j$).

With the substitution of $K$ for $K(\ell)$,
the summands in 
\eqref{cases.intersection} no longer depend on 
$\ub$ or $\ell$, and, using \Cref{prop:b.bd},
we have a simpler version of \Cref{beta.def}:

\begin{cor}\label{beta.def'} 
For any $e \in H$ and $\ud$, 
the number of $\hat R \in \hat \cR_{\ud}$ 
that contain $e$ and fit $H$ is at most
\beq{cases.intersection'} 
\gb(\ud):= 
2(j+e^{o(j)})(j+1)K\cdot\prod_{\text{$i$ small}} (\sqrt \eps d)^{f_i} \prod_{\text{$i$ big}} d_i^{f_i}.
\enq
\end{cor}

So, since for any $R \in \cR_\ud$,
\beq{deg.sum} 
e_R=\sum\{ d_i :i \in [0, j]\} \le D+j,
\enq
we may overestimate the number of
$\hat R$'s in $\hat \R_\ud$ that fit $H$ 
and share edges with a given $\hat R_0$ by
$e_R \gb(\ud) $, 
which with \Cref{Lnubd} and \eqref{deg.sum} gives
\beq{eq:upshot} 
\sum_{R \in \cR_{\ud}} N^*(H,R) ~<~
\sum_{R \in \cR_{\ud}} \nu(H,R) \cdot e_R \cdot \beta(\ud) ~<~
\ga(\ud) (D+j) \beta(\ud).
\enq
Now inserting the values of $\ga(\ud)$ from 
\eqref{bd1.1'} and $\gb(\ud)$ from 
\eqref{cases.intersection'} (with the bound on $K$
in \eqref{d.b.bd}), and (slightly) simplifying,
we find that the l.h.s.\ of
\eqref{eq:upshot} is at most
\beq{total.sum}
n \cdot\left\{e(D+j)^2 
2(j+ e^{o(j)})(j+1)e^{O(j)/\sqrt \eps d} \right\}\cdot
\prod_{\text{$i$ small}} (\eps^{3/2}d^2)^{f_i}
\prod_{i \text{ big}} \left[\left(\frac{e\eps d}{d_i}\right)^{d_i} d_i^{2f_i}\right].
\enq
This looks unpleasant but is actually simple, since 
the terms
$(\frac{e\eps d}{d_i})^{d_i}$
dominate the rest (apart from $n$): since $d_i\geq \sqrt{\eps} d$ when $i$ is big, 
the product of these terms is less than
\beq{main.term}
(e\sqrt{\eps})^{D},
\enq
whereas:
the expression in $\{\,\}$'s is 
$O(D^2)e^{O(j)}$;
since $\sum f_i = j$, the first 
product (even sacrificing
the terms with $\eps^{3/2}$) is at most
$d^{2j}$; 
and what's left of the second product is 
\[\prod_{i \text{ big}} d_i^{2f_i} \leq \prod_{i \text{ big}} d_i^{2\gD} < 2^{2D\gD}\]
(using $d_i < 2^{d_i}$
for the second inequality). 
So the bound in \eqref{total.sum} is 
no more than
\[
n D^2 e^{O(j)} d^{2j} 2^{2D\gD} (e\sqrt{\eps})^{D} < n 2^{O(D)}(e\sqrt{\eps})^{D},
\]
where the inequality (finally)
uses
$D > j \log d$ (and the implied
constant in $2^{O(D)}$ depends on $\gD$).

Finally, now fixing $D > j \log d$
and letting $\ud$ vary,
and recalling from 
\eqref{count.bd} that the number of 
$\ud$'s with $D(\ud) = D$ is
\[\exp\left[O\left(\frac{D}{\sqrt \eps d} 
\log (d)\right)\right] = 2^{O(D)},\]
we have
\[
\sum_{D(\ud) = D}\sum_{R \in \cR_{\ud}} N^*(H,R) < n 2^{O(D)}(e\sqrt{\eps})^{D},
\]
which (with a small enough $\eps$) gives 
\eqref{2nd.goal}.
\qed

\section{Cycles}\label{sec:cycles}

Here we prove the cycle portion of \Cref{thm:clique}; 
to repeat:
We assume that $p=Lq$ with $L$ a large constant,
$H$ is $q$-sparse,
and $F=C_k$ for some $k\in [3,n]$, 
and want to show
\beq{cycle.show}
N(H,F) < \E_pX_F.
\enq
Since
\beq{eq:EpXF.lb'} \EE_pX_{F}=N(K_n, F)p^{e(F)}=\frac{(n)_k}{2k}p^k > \frac{1}{2k}\left(\frac{np}{e}\right)^k =\frac{1}{2k}\left(\frac{L}{e}\right)^k (nq)^k \ge L^{.9k}(nq)^k,
\enq
it is enough to show that $N(H,F)$ is at most the r.h.s.\ of \eqref{eq:EpXF.lb'}.
To begin we eliminate easy ranges for $q$:

\begin{prop}\label{prop:cycle.density} 
If $q\not\in (1/n ,  1/\sqrt n)$ then \eqref{cycle.show} holds.
\end{prop}

\begin{proof} If $q \le 1/n$, then $\EE_qX_F < n^kq^k \le1$; so $q$-sparsity of $H$ forces 
$N(H,F)=0$.  (Note that
when $q<1/(3n)$, the second part of
\Cref{MT} gives \Cref{conj:main} for a 
\emph{general} $F$.)

For $q\ge 1/\sqrt{n}$ we use the naive bound
\beq{eq:cycle.dense.ub}
N(H,F)\le 
e_H \cdot \Delta_H^{k-2}, 
\enq
in which the r.h.s.\ (over)counts
ways to choose 
$xy \in E(H)$ and a $(k-1)$-edge path (in $H$)
joining $x$ and $y$.
We then recall that 
\Cref{prop:max.deg} promises 
$\gD_H\le 2enq$, while 
\Cref{cor:e_H} bounds $e_H$ by 
$ n\log n$ in general, and by $3n$
if $q < (\log n/n)^{1/2}$; so
\[
N(H,F) \le \left\{\begin{array}{ll} 
n\log n(2enq)^{k-2} &\mbox{in general,}\\
3n(2enq)^{k-2}&\mbox{if $q< (\log n/n)^{1/2}$,}
\end{array}\right.
\]
and the r.h.s.\ of \eqref{eq:EpXF.lb'}
exceeds the first bound if 
$q \ge (\log n/n)^{1/2}$ and 
the second if $q\ge 1/\sqrt n$.
\end{proof}

So for the rest of this discussion we assume 
\beq{q.bound.cycle}
1/n<q<1/\sqrt n.
\enq

Our approach here is simple
(the trivial 
\eqref{eq:cycle.dense.ub} is a first example), 
but 
turns out to be rather delicate:
for some carefully chosen $m$ we use \Cref{thm:tree}
to bound the number of $P_m$'s 
($m$-edge paths) in $H$, 
and then bound the number of extensions to copies of 
$F$ using \Cref{lem:cycle.key},
which is the main new point in this section. 
What we need from \Cref{thm:tree} is
\beq{eq:paths.ub} 
\text{there is a fixed $L_1$ such that if $H$ is $q$-sparse then, for any $m$, 
$N(H, P_m)< n^{m+1}(L_1 q)^m$.}
\enq

\mn

For the rest of this discussion
we work with the following 
definitions and assumptions. 
We assume 
\[
nq=n^{c}
\]
(so, by \eqref{q.bound.cycle}, 
$c\in (0,1)$).
For distinct $x,y \in V(H)$, we use 
"$(x,y)$-$P_\ell$" for a $P_\ell$ in $H$ with 
endpoints $x$ and $y$, and set
\beq{gam.def}
\gamma(\ell)=\max_{\substack{x,y \in V(H)\\ x \ne y}} |\{\mbox{$(x,y)$-$P_\ell$'s} \}|. 
\enq 
For $\gd \in (0,1)$, we define 
$\hat \ell(\gd)$ to 
be the largest integer $\ell$ for which
\beq{eq:ell.def}
(nq)^\ell<n^{1-\gd c}. 
\enq 
\begin{lem}\label{lem:cycle.key}
Suppose $H$ is $q$-sparse. 
Let $\gd \in (0,1)$ be given and $\hat \ell=\hat \ell(\gd)$.
If $\ell$ satisfies \eqref{eq:ell.def}
(i.e.\ $\ell\le \hat\ell$), 
then $\gamma(\ell)=O(\hat \ell/\gd)$,
and if
\beq{eq:ell.def'}
(nq)^\ell<n^{1-\gd},\enq
then $\gamma(\ell)=O(1/\gd)$.
\end{lem}

\begin{proof}
For this discussion we fix distinct $x,y \in V(H)$. 
We usually use $K$, often subscripted, for 
$(x,y)$-$P_\ell$'s, and $\gc =\gc(x,y)$ 
for the number of these;  
so we should show $\gamma=O(\hat \ell/\gd)$
if \eqref{eq:ell.def} holds and 
$\gamma=O(1/\gd)$ if we assume
\eqref{eq:ell.def'}.
We will often 
treat $K$'s as sets of \emph{edges}.

Choose $K_0, K_1, \ldots, K_m$ so that
\beq{eq:K_i.minimal} 
e_i:=
|K_i \setminus \left(\cup_{j<i} K_j\right)|
=\min\{|K \setminus \cup_{j <i} K_j|:
K \not\sub \cup_{j <i} K_j\} \quad \forall i \ge 1
\enq
and 
\beq{eq:covering} 
\mbox{$R ~:= ~\bigcup_{i=0}^m K_i\,\,$ 
contains all $\xypl$'s.} 
\enq
We use $R_i$ for the subgraph of $H$ 
consisting of (the edges of) 
$K_i \setminus (\cup_{j < i} K_j)$ and their vertices
(e.g.\ $R_0=K_0$),
and set 
\[
v_i=|V(R_i)\sm V(\cup_{j<i}R_j)|.
\]
Thus 
\beq{viei}
\mbox{$v_0=e_0 +1=\ell +1\,\,$ and 
$\,\,v_i\leq e_i-1 \leq \ell -1\,\,$ for $i\in [m]$}
\enq
(since for $i\in [m]$, $E(R_i)$ consists of 
edge-disjoint paths with ends in $V(\cup_{j<i}R_j)$),
which, with the $q$-sparsity of $H$, gives
\beq{eq:E.hat.R} 
1\le \EE_qX_{R}< n^{v_R}q^{e_R}=n^{\ell+1}q^{\ell}
\prod_{i=1}^m n^{v_i}q^{e_i}
\le 
n^{\ell+1}q^{\ell}\prod_{i=1}^m  [n^{-1}(nq)^{e_i}].
\enq

The lemma will follow from 
Claims~\ref{cl:cl1}-\ref{cl:cl3}; the first 
of these bounds $m$, and others bound the 
number of $(x,y)$-$P_\ell$'s not in $\{K_0\dots K_m\}$.

\begin{claim}\label{cl:cl1}
If \eqref{eq:ell.def} holds then
$m=O(\hat \ell/\gd)$, and if 
\eqref{eq:ell.def'} holds then
$m<2/\gd$.
\end{claim}

\begin{subproof}
For the second part just notice that   
\eqref{eq:ell.def'} bounds the r.h.s.\ of
\eqref{eq:E.hat.R} by 
\[
n^{\ell+1}q^{\ell}(n^{-1}(nq)^{\ell})^m
<n^2\cdot n^{-\gd m}.
\]

The first part will follow from density considerations.
We may rewrite \eqref{eq:ell.def} as 
$q<n^{-(\ell-1+\gd c)/\ell}$, which with the
the second bound in 
\Cref{prop:dense} 
(that is, $m(H)<1/c$ if $q\le n^{-c}$)
gives
\beq{eq:e/v.ub}
\frac{v_R}{e_R} \ge \frac{\ell-1+\gd c}{\ell}.
\enq
But $e_R=\sum e_i$ and, in view of
\eqref{viei}, 
\[
v_R =\sum v_i \,\, 
\le \,\, e_0+1 +\sum_{i=1}^m(e_i-1) 
\,\,  = \,\,  e_R-m+1;
\]
so with \eqref{eq:e/v.ub} we have 
\[
\frac{\ell-1+\gd c}{\ell} \le
\frac{e_R-m+1}{e_R},
\]
which, with $e_R \le \ell(m+1)$
(and a little rearranging),
gives
\[
m\le 2/(\gd c)-1.
\]
The claim follows since 
$n^{c(\hat \ell+1)} \ge n^{1-\gd c}$
(by the maximality of $\hat \ell$)
and $c<1/2$ (by \eqref{q.bound.cycle}) 
give 
$\hat \ell =\gO(1/c)$.
\end{subproof}

\begin{claim}\label{cl:cl2}
If $K\not\in \{K_0\dots K_m\}$,
then there is $i \in [0,m]$ such that $|K \cap K_i|\ge \ell/8$.
\end{claim}

\begin{subproof}
Let $j_0$ and $j_1$ (possibly equal) be minimum 
with $|K \setminus (\cup_{j \le j_0} K_j)| \le \ell/2$ and $K \sub (\cup_{j \le j_1} K_j)$
(with existence given by \eqref{eq:covering}).
Then
$|K \cap (\cup_{j \in [j_0, j_1]} K_j)| >\ell/2$, so 
\beq{eq:exist} 
\text{there is $i \in [j_0, j_1]$ such that $|K \cap K_i| \ge \ell/(2(j_1-j_0+1))$.}
\enq
This immediately gives the claim
if $j_0=j_1$. For $j_0<j_1$, we return to
\eqref{eq:E.hat.R},
observing that 
(justification to follow) \beq{eq:R_i} 
n^{-1}(nq)^{e_i}
\le \left\{\begin{array}{ll}
n^{-1}(nq)^{\ell}
<n^{-\gd c}<1 &
\mbox{for any $i\in [m]$,}\\
n^{-1}(nq)^{\ell/2}
<n^{-1}n^{(1-\gd c)/2}<n^{-1/2}
&\mbox{if $i\in [j_0+1,j_1]$.}
\end{array}\right.
\enq
Here both lines use 
\eqref{eq:ell.def},
and---the main point---the second uses 
\[
e_i\leq \ell/2,
\]
which holds since 
otherwise \eqref{eq:K_i.minimal} would have
forbidden choosing $K_i$ when we could have
chosen $K$.

The $q$-sparsity of $H$, with \eqref{eq:E.hat.R},
\eqref{eq:R_i} and, again, \eqref{eq:ell.def},
now gives
\beq{1EqXR}
1 \le \E_qX_R <  n^{\ell+1}q^\ell \cdot n^{-(j_1-j_0)/2}
<n^{2-(j_1-j_0)/2};
\enq
so $j_1-j_0\leq 3$ and 
\eqref{eq:exist} completes the proof.
\end{subproof}

For the next claim, to avoid confusion, 
we use $Q$ in place of $K$ for 
$\xypl$'s.
\begin{claim}\label{cl:cl3}
For any $Q_0$, the number of  
$Q$'s sharing at least $\ell/8$ edges with $Q_0$
is  $O(1)$.
\end{claim}
\begin{subproof}
Choose $Q_1, \ldots, Q_m$ so that
\[    
\mbox{$|Q_i\cap Q_0|\ge \ell/8\,\,\,$ and 
$\,\,\, Q_i \not\sub \cup_{j <i} Q_j 
\quad \forall i \in [m]$}
\]    
and 
\[   
\mbox{$R ~:= ~\bigcup_{i=0}^m Q_i\,\,$ 
contains all $Q$'s with $Q\cap Q_0\geq \ell/8$.} 
\]   
We again use $R_i$ for the subgraph of $H$ 
consisting of the edges of 
$Q_i \setminus (\cup_{j < i} Q_j)$ and their vertices, and for $i\in [m]$ set
\[
\mbox{$e_i=
|Q_i \setminus \left(\cup_{j<i} Q_j\right)|,\,\,$ 
$\,v_i=|V(R_i)\sm V(\cup_{j<i}R_j)|$,}
\]
and 
\beq{f(i)}
f(i) = |\{(v,e):e\in Q_i\sm (\cup_{j<i}Q_j),
v\in e\cap V(\cup_{j<i}R_j)\}|.
\enq
The main point here is
\beq{f(i)main}
\sum f(i)=O(1).
\enq
(Arguing as for
\Cref{cl:cl2} 
gives $m=O(1)$, but we now need 
a little more.)

Here we observe that, 
for each $i$, $R_i$ is an edge-disjoint union of (say) $a_i$ 
paths, each of which shares precisely its
endpoints with $\cup_{j<i} R_j$. 
Each of these paths contributes (exactly) two 
pairs $(v,e)$ to $f(i)$, and each $(v,e)$
counted by $f(i)$ arises in this way; 
so 
\beq{fiai}
f(i)=2a_i.
\enq

\begin{subproof}[Proof of \eqref{f(i)main}.]
Noting that
\[
\mbox{$v_i=e_i-a_i\,\, $ and $\,\, e_i\le 7\ell/8$,}
\]
we have (more or less as in
\eqref{eq:E.hat.R}-\eqref{1EqXR})
\[   
1 \le \EE_qX_{R} < n^{v_R}q^{e_R}=n^{\ell+1}q^{\ell}
\prod_{i=1}^m n^{v_i}q^{e_i} 
< n^2\prod_{i=1}^m n^{v_i}q^{e_i}
\]   
and 
\[
n^{v_i}q^{e_i} = n^{-a_i}(nq)^{e_i}
\le 
n^{-a_i}n^{7(1-\gd c)/8},
\]
which with \eqref{fiai} easily give \eqref{f(i)main}. 
\end{subproof}

Now, with $v$ running over $V(R)\sm \{x,y\}$ 
(so $d_R(v)\ge 2$), we have
\beq{dRdR}
(d_R(x)+d_R(y)-2) +\sum_v(d_R(v)-2) = 2\sum a_i
\enq
(since if $\gS_i$ is the l.h.s.\ of 
\eqref{dRdR} with $R$ replaced by $\cup_{j\le i}Q_j$,
then $\gS_0=0$ and $\gS_i-\gS_{i-1}= 2a_i$ 
for $i\geq 1$).
Combined with \eqref{f(i)main} (and \eqref{fiai}) this  gives 
\Cref{cl:cl3}, since any 
$\xypl$ in $R$ is determined by 
what it does at $x$, $y$ and 
the $v$'s of degree greater than 2. \end{subproof}

Finally, 
the combination of \Cref{cl:cl2} and \Cref{cl:cl3}
(used with $Q_0=K_i$ for $i\in [0,m]$) 
bounds the number of $(x,y)$-$P_\ell$'s by
$O(m)$; and adding the bounds on $m$ from
\Cref{cl:cl1} then gives \Cref{lem:cycle.key}.
\end{proof}

With \eqref{eq:paths.ub} and 
\Cref{lem:cycle.key} in hand, 
we return to \Cref{thm:clique},
setting (for the rest of our discussion)
\beq{ell0.1}
\dhl = \hat\ell(0.1).
\enq

\nin
To begin, we observe that the 
theorem is easy when $k$ is fairly small
(here we don't need \eqref{eq:paths.ub}):

\begin{lem}\label{lem:small k}
If $k \le \dhl +1$, then $N(H,F)< \EE_pX_F.$
\end{lem}

\begin{proof} 
We again use the approach sketched following
\Cref{prop:nu}, beginning by observing that
\[
N(H,F)\leq \nu(H,F)\cdot k \cdot \gamma(k-1),
\]
since each of the $k$ edges of 
a given copy of $F$ is contained in
fewer than $\gamma(k-1)$ other copies.
(Recall $\gc$ and $\nu$ were defined in
\eqref{gam.def} and 
following \Cref{cor:e_H}.)

Since $\nu(H,F) \le e\EE_qX_F$ 
(see \Cref{prop:nu}), the lemma will follow 
if we show
\beq{k-1,k}
\gamma(k-1)=O(k),
\enq
since then 
\[
N(H,F)\le e\EE_qX_F\cdot O(k^2)<
L^k\EE_qX_F=\EE_pX_F. 
\]

\begin{subproof}[Proof of \eqref{k-1,k}]
This is two applications of 
\Cref{lem:cycle.key}:
if $k \le \dhl/2$, then
\[
(nq)^{k-1}\le (nq)^{\dhl/2}
\le n^{(1-0.1 c)/2} <n^{1/2},
\]
so the second part of the lemma gives
$\gamma(k-1)=O(1)$; and if 
$k\in [\dhl/2, \dhl+1]$, then
\[
(nq)^{k-1}\le (nq)^{\dhl}<n^{1-0.1 c}
\]
and the first part of the lemma gives 
$\gamma(k-1)=O(\dhl)=O(k)$.
\end{subproof}

\end{proof}

For the rest of this section, we assume 
\beq{eq:kbig}
k\ge \dhl+2,
\enq 
and divide the argument according to
the value of $q\,$
($\in (1/n ,  1/\sqrt n)$; see 
\eqref{q.bound.cycle}).

\mn
\emph{Small $q$.} We first assume
\[    
1/n < q \le \log n/n.
\]   
(The upper bound 
could be considerably relaxed.)

Setting $m=k-\dhl~( \ge 2)$, we have
\[
N(H,F)\le N(H, P_{m})\cdot\gamma(\dhl)
\]
(since $\gamma(\dhl)$ bounds the 
number of completions of any given 
$P_m$ in $H$ to a $C_k$); and
inserting the bounds from 
\eqref{eq:paths.ub} and \Cref{lem:cycle.key}
(namely, 
$N(H,P_{m})< n^{m+1}(L_1q)^{m}$
and $\gamma(\dhl) = O(\dhl)$),
and letting $L= L_1^2$ (and $p=Lq$), gives 
\[
N(H,F) < n^{m+1}(L_1q)^{m}\cdot O(\dhl) 
\leq L^{k/2} n(nq)^{m}\cdot O(\dhl).
\]
To bound the r.h.s.\ of this, 
we first observe that
maximality of $\dhl\,$ ($=\hat\ell(0.1)$) gives
\beq{eq:hat.ell.plus2} (nq)^{\dhl+2}=n^c(nq)^{\dhl+1}\ge 
n^c n^{1-0.1 c}>n.\enq
So 
$n(nq)^m<(nq)^{\dhl+2}(nq)^m=(nq)^{k+2}$,
and $N(H,F)$ is less than 
\[
L^{k/2}(nq)^{k}O((nq)^2\dhl)<L^{.9k}(nq)^k<\EE_pX_F,
\]
where the first inequality uses
the easy 
\beq{kdhlgO}
k> \dhl =\gO(\log n/\log\log n)
\enq
(see \eqref{ell0.1} and 
\eqref{eq:ell.def}; here  
$\dhl >\log\log n$ would suffice),
and the second is \eqref{eq:EpXF.lb'}.
\qed

\mn
\emph{Large $q$.}
Here we are in the complementary range
\[
\log n/n < q < 1/ \sqrt n.
\]
We again use 
\beq{eq:factor}
    N(H,F) \leq N(H,P_m)\cdot\gc(\ell'),
\enq
with a suitable $\ell'$ and $m = k - \ell'$. 
In this case 
we bound the first factor by the 
trivial 
\[
N(H,P_m) \leq 2e_H\gD_H^{m-1}
\]
(cf.\ \eqref{eq:cycle.dense.ub};
curiously this now does better than \eqref{eq:paths.ub}), 
which with 
\Cref{cor:e_H} and \Cref{prop:max.deg} gives
\beq{NHPm}
N(H,P_m) < O(n (2enq)^{k - \ell' - 1}).
\enq
So we will mainly be interested in $\gc(\ell')$.

Let $\ell' $ be the largest integer $\ell$
satisfying 
$ n^{\ell - 1} q^\ell \leq L^{-1/4}$,
noting that 
\beq{ell'lb}
\ell' + 1 > (1 - o(1))\log n/\log(nq)
\enq
(since $n^{\ell'}q^{\ell' + 1} > L^{-1/4}$),
and set
\[
f = f(n) = 1/(n^{\ell'-1}q^{\ell'})
\]
and 
\[
\gdd =\log f/\log (L^{1/4}nq);
\]
noting that
$L^{1/4} \leq f < L^{1/4}nq$ implies
$\gd\in (0,1)$ and
\beq{gd.bds}
\log f/\log(nq)> \gdd
> (1-o(1))\log f/\log(nq)
\enq
(the latter since here $nq=\go(1)$).

We first check that
\Cref{lem:cycle.key}, used with 
$\ell=\ell'$ (and $\gd = \gdd$), gives
\beq{eq:gambound}
\gc(\ell') = O(\log n/ \log f).
\enq

\begin{subproof} 
The upper bound in \eqref{gd.bds} implies
\eqref{eq:ell.def} in the 
form
$(nq)^{\gdd} 
< (n^{\ell' - 1}q^{\ell'} )^{-1} = f$ 
(recall $nq=n^c$); 
so (the first part of)
\Cref{lem:cycle.key} gives $\gc(\ell') = O(\hat\ell(\gdd)/\gdd)$, and
\eqref{eq:gambound} then follows from 
the lower bound in \eqref{gd.bds} and
the trivial
$
\hat \ell (\gdd)= 
O(\log n/\log (nq)).
$   
\end{subproof}

We should also note that 
($m :=$) $k - \ell'\geq 1$,
which is given by \eqref{eq:kbig}
since 
$n^{\dhl+ 1}q^{\dhl + 2} 
\ge n^{-0.1c}nq > 1 > L^{-1/4}$
implies $\ell' \leq \dhl +1$.
We may thus insert
\eqref{NHPm} and 
\eqref{eq:gambound} in \eqref{eq:factor}, yielding
\[
N(H,F)= O(n (2enq)^{k - \ell' - 1}\log n/ \log f),
\]
which, for large enough $L$, is less than
\[(n^kq^k L^{.9k}) \cdot \left(f\log n/(L^{k/2}nq\log f)\right).\]
In view of \eqref{eq:EpXF.lb'}, 
it is thus enough to show 
$(f/\log f)\log n < L^{k/2}nq$, 
which, rewritten as
\[
(f/\log f)(\log n/\log (nq)) < L^{k/2} nq/\log(nq),
\]
is true because
$f/\log f < L^{1/4}nq /\log(nq)$ 
(since $f < L^{1/4}nq$ and
$x/\log x$ is increasing for $x\ge e$)
and, with plenty of room,  
$\log n/\log(nq) < L^{k/4}$ follows from \eqref{ell'lb}.

\section*{Acknowledgments}
QD was supported by NSF Grants DMS-1954035 and DMS-1928930. JK was supported by NSF Grants DMS-1954035 and DMS-2452069.  
JP was supported by NSF Grant DMS-2324978, NSF CAREER Grant DMS-2443706 and a Sloan Fellowship.

\bibliographystyle{plain}
\bibliography{bibliography}

\end{document}